\documentclass[12pt,reqno]{amsart}
\usepackage[lmargin=1in,rmargin=1in,tmargin=1in,bmargin=1in]{geometry}
\usepackage{graphicx, float, epstopdf}
\usepackage{hyperref, color}
\usepackage{amsfonts,amssymb,amsmath,amsthm}
\usepackage{url, multirow}

\newcommand{\df}{\dfrac}
\newcommand{\tf}{\tfrac}

 \renewcommand{\arctan}{\tan^{-1}}
 \renewcommand{\a}{\alpha}
\renewcommand{\b}{\beta}

\newcommand{\G}{\Gamma}

\renewcommand{\(}{\left\(}
\renewcommand{\)}{\right\)}
\let\dotlessi=\i
\renewcommand{\i}{\infty}

\numberwithin{equation}{section}
\theoremstyle{plain}
\newtheorem{theorem}{Theorem}[section]
\newtheorem{lemma}[theorem]{Lemma}
\newtheorem{corollary}[theorem]{Corollary}

\newtheorem{definition}[theorem]{Definition}

\allowdisplaybreaks[4]

   \makeatletter
\def\proof{\@ifnextchar[{\@oproof}{\@nproof}}
\def\@oproof[#1][#2]{\trivlist\item[\hskip\labelsep\textit{#2 Proof of\
#1.}~]\ignorespaces}
\def\@nproof{\trivlist\item[\hskip\labelsep\textit{Proof.}~]\ignorespaces}

\makeatother

\begin{document}
\title[Koshliakov kernel and the Riemann zeta function]{Koshliakov kernel and identities involving the Riemann zeta function}
\author{Atul Dixit}
\address{Department of Mathematics, Tulane University, New Orleans, LA 70118, USA}
\email{adixit@tulane.edu}
\author{Nicolas Robles}
\address{Institut f\"{u}r Mathematik, Universit\"{a}t Z\"{u}rich, Winterthurerstrasse 190, CH-8057 Z\"{u}rich, Switzerland} \email{nicolas.robles@math.uzh.ch}

\author{Arindam Roy}
\address{Department of Mathematics, University of Illinois, 1409 West Green
Street, Urbana, IL 61801, USA} \email{roy22@illinois.edu}
\author{Alexandru Zaharescu}
\address{Department of Mathematics, University of Illinois, 1409 West Green
Street, Urbana, IL 61801, USA and Institute of Mathematics of the Romanian
Academy, P.O.~Box 1-764, Bucharest RO-70700, Romania} \email{zaharesc@illinois.edu}
\maketitle
\begin{abstract}
Some integral identities involving the Riemann zeta function and functions reciprocal in a kernel involving the Bessel functions $J_{z}(x), Y_{z}(x)$ and $K_{z}(x)$ are studied. Interesting special cases of these identities are derived, one of which is connected to a well-known transformation due to Ramanujan, and Guinand.
\end{abstract}
\section{introduction}\label{intro}

In their long memoir \cite[p.~158, Equation (2.516)]{hl}, Hardy and Littlewood obtain, subject to certain assumptions unproved as of yet (for example, the Riemann Hypothesis), an interesting modular-type transformation involving infinite series of M\"{o}bius function as suggested to them by some work of Ramanujan. By a modular-type transformation, we mean a transformation of the form $F(\alpha)=F(\beta)$ for $\a\b=\text{constant}$. On pages $159-160$, they also give a generalization of the transformation for any pair of functions reciprocal to each other in the Fourier cosine transform as indicated to them by Ramanujan. 

Let $\Xi(t)$ be Riemann's $\Xi$-function defined by \cite[p.~16]{titch}
\begin{equation}\label{xif}
\Xi(t):=\xi(\tf{1}{2}+it),
\end{equation}
where
\begin{equation}\label{xi}
\xi(s):=\frac{1}{2}s(s-1)\pi^{-\tf{s}{2}}\Gamma\left(\frac{s}{2}\right)\zeta(s)
\end{equation}
is the Riemann $\xi$-function \cite[p.~16]{titch}. Here $\G(s)$ is the gamma function \cite[p.~255]{as} and $\zeta(s)$ is the Riemann zeta function \cite[p.~1]{titch}. 

A natural way to obtain similar such modular-type transformations is by evaluating integrals of the type
\begin{equation*}
\int_{0}^{\infty}f(t)\Xi(t)\cos\left(\frac{1}{2}t\log\alpha\right)\, dt,
\end{equation*}
where $f(t)=\phi(it)\phi(-it)$ for some analytic function $\phi$, since they are invariant under $\a\to 1/\a$, although the aforementioned transformation involving series of M\"{o}bius function is not obtainable this way. Ramanujan studied an interesting integral of this type in \cite{riemann}.

Motivated by Ramanujan's generalization, the authors of the present paper, in \cite{drrz01}, studied integrals of the above type but with the cosine function replaced by a general function $Z\left(\frac{1}{2}+it\right)$, which is an even function of $t$, real for real $t$, and depends on the functions reciprocal in the Fourier cosine transform. Several integral evaluations such as the one connected with the general theta transformation formula \cite[Equation (4.1)]{ingenious}, and those of Hardy \cite[Equation (2)]{ghh} and Ferrar \cite[p.~170]{ingenious} were obtained in \cite{drrz01} as special cases by evaluating these general integrals for specific choices of $f$.

Ramanujan \cite{riemann} also studied integrals of the form
\begin{equation*}
\int_{0}^{\infty}f\left(\frac{t}{2}, z\right)\Xi\left(\frac{t+iz}{2}\right)\Xi\left(\frac{t-iz}{2}\right)\cos \left(\frac{1}{2} t\log\alpha\right)\, dt,
\end{equation*}
where 
\begin{equation}\label{fdefi}
f(t, z)=\phi(it, z)\phi(-it, z),
\end{equation}
with $\phi$ being analytic in the complex variable $z$ and in the real variable $t$.
With $f$ being of the form just discussed, in the present paper, we study a generalization of the above integral of the form 
\begin{equation}\label{genintgen}
\int_{0}^{\infty}f\left(\frac{t}{2}, z\right)\Xi\left(\frac{t+iz}{2}\right)\Xi\left(\frac{t-iz}{2}\right)Z\left(\frac{1+it}{2}, \frac{z}{2}\right)\, dt,
\end{equation}
where the function $Z\left(\frac{1}{2} +it, z\right)$ depends on a pair of functions which are reciprocal to each other in the kernel
\begin{equation}\label{kernel}
\cos \left( {\frac{{\pi z}}{2}} \right){M_z}(4\sqrt {x} ) - \sin \left( {\frac{{\pi z}}{2}} \right){J_z}(4\sqrt {x} ),
\end{equation}
where 
\begin{equation*}
{M_z }(x) := \frac{2}{\pi }{K_z }(x) - {Y_z }(x).
\end{equation*}
Here $J_{z}(x)$ and $Y_{z}(x)$ are Bessel functions of the first and second kinds respectively, and $K_{z}(x)$ is the modified Bessel function. 

We call this kernel as the \emph{Koshliakov kernel} since Koshliakov \cite[Equation 8]{kosh1938} was the first mathematician to construct a function self-reciprocal in this kernel, namely, he showed that for real $z$ satisfying $-\tfrac{1}{2}<z<\tfrac{1}{2}$,
\begin{equation}\label{koshlyakov-1}
\int_{0}^{\infty} K_{z}(t) \left( \cos(\pi z) M_{2z}(2 \sqrt{xt}) -
\sin(\pi z) J_{2z}(2 \sqrt{xt}) \right)\, dt = K_{z}(x).
\end{equation}
(It is easy to see that this formula actually holds for complex $z$ with $-\tfrac{1}{2}<$ Re$(z)<\tfrac{1}{2}$.) Dixon and Ferrar \cite[Equation (1)]{dixfer3} had previously obtained the special case $z=0$ of the above integral evaluation. 

The Koshliakov kernel occurs in a variety of places in number theory, for example, in the extended form of the Vorono\"{\dotlessi} summation formula \cite[Theorems 6.1, 6.3]{bdrz}. In view of Koshliakov's aforementioned work, the integral transform 
\begin{equation*}
\int_{0}^{\i} g(t,z) \left( \cos(\pi z) M_{2z}(2 \sqrt{xt}) -
\sin(\pi z) J_{2z}(2 \sqrt{xt}) \right)\, dt,
\end{equation*}
where $g(t,z)$ is a function analytic in the real variable $t$ and in the complex variable $z$, is named in \cite[p.~70]{bdrz} as the \emph{first Koshliakov transform} of $g$. It arises naturally when one considers a function corresponding to the functional equation of an even Maass form in conjunction with Ferrar's summation formula; see the work of Lewis and Zagier \cite[p.~216--217]{lewzag}, for example.

Let the functions $\varphi$ and $\psi$ be related by
\begin{align} \label{recip1}
\varphi (x, z) = 2\int_0^\infty  {\psi (t, z)\left( {\cos \left( {\pi z} \right){M_{2z}}(4\sqrt {tx} ) - \sin \left( \pi z \right){J_{2z}}(4\sqrt {tx} )} \right)\, dt},
\end{align}
and 
\begin{align} \label{recip2}
\psi (x, z) = 2\int_0^\infty  {\varphi (t, z)\left( {\cos \left( {\pi z} \right){M_{2z}}(4\sqrt {tx} ) - \sin \left( \pi z \right){J_{2z}}(4\sqrt {tx} )} \right)\, dt} .
\end{align}
Moreover, define the normalized Mellin transforms $Z_1(s, z)$ and $Z_2(s, z)$ of the functions $\varphi(x, z)$ and $\psi(x, z)$ by
\begin{align}
\Gamma \left( {\frac{s-z}{2}} \right)\Gamma \left( {\frac{s+z}{2}} \right){Z_1}(s, z) &= \int_0^\infty  {x^{s - 1}}{\varphi (x, z)\, dx},\label{defZ1}\\
 \Gamma \left( {\frac{s-z}{2}} \right)\Gamma \left( {\frac{s+z}{2}} \right){Z_2}(s, z) &= \int_0^\infty  {x^{s - 1}}{\psi (x, z)\, dx},\label{defZ2}
\end{align}
where each equation is valid in a specific vertical strip in the complex $s$-plane. Set 
\begin{equation}\label{add}
Z(s, z):= Z_1(s, z) + Z_2(s, z) \quad \textnormal{and} \quad \Theta(x, z) := \varphi(x, z) + \psi(x, z),
\end{equation}
so that
\begin{align}\label{zthph}
\Gamma \left( {\frac{s-z}{2}}\right)\Gamma \left( {\frac{s+z}{2}}\right)Z(s, z) = \int_0^\infty  {{x^{s - 1}}\Theta (x, z)\, dx} 
\end{align}
for values of $s$ in the intersection of the two vertical strips.

In this paper, we evaluate the integrals in \eqref{genintgen} for two specific choices of $f(t, z)$ satisfying \eqref{fdefi}. The function $Z\left(\frac{1}{2}+it,z\right)$ in these integrals depends on the functions $\varphi(x,z)$ and $\psi(x,z)$ satisfying \eqref{recip1} and \eqref{recip2}, and belonging to the class $\Diamond_{\eta, \omega}$ defined below.
\begin{definition}
Let $0<\omega\leq \pi$ and $\eta>0$. For a fixed $z$, if $u(s, z)$ is such that
\begin{enumerate}
\item[i)] $u(s, z)$ is analytic of $s=re^{i\theta}$ regular in the angle defined by $r>0$, $|\theta|<\omega$,
\item[ii)] $u(s, z)$ satisfies the bounds
\begin{equation}\label{growth}
u(s, z)=
			\begin{cases}
			O_{z}(|s|^{-\delta}) & \mbox{ if } |s| \le 1,\\
			{O_z(|s|^{-\eta-1-|\textup{Re}(z)|})} & \mbox{ if } |s| > 1,
			\end{cases}
\end{equation}
\end{enumerate}
for every positive $\delta$ and uniformly in any angle $|\theta|<\omega$, then we say that $u$ belongs to the class $\Diamond_{\eta, \omega}$ and write $u(s, z)\in \Diamond_{\eta,\omega}$.
\end{definition}
We are now ready to state our two main theorems.
\begin{theorem}\label{ramguigene}
Let $\eta>1/4$ and $0<\omega\leq \pi$. Suppose that $\varphi,\psi \in \Diamond_{\eta,\omega}$, are reciprocal in the Koshliakov kernel as per \eqref{recip1} and \eqref{recip2}, and that $-1/2<$ \textup{Re}$(z)<1/2$. Let $Z(s, z)$ and $\Theta(x, z)$ be defined in \eqref{add}. Let $\sigma_{-z}(n)=\sum_{d|n}d^{-z}$. Then,
\begin{align}\label{ramguigeneid}
  &\frac{32}{\pi}\int_0^\infty  {\Xi \left( {\frac{{t + iz}}{2}} \right)\Xi \left( {\frac{{t - iz}}{2}} \right)Z\left( {\frac{{1 + it}}{2}}, \frac{z}{2} \right)\frac{{dt}}{{({t^2}+(z+1)^2)({t^2}+(z-1)^2)}}}  \nonumber \\
  &= \sum\limits_{n = 1}^\infty  {{\sigma _{ - z}}(n){n^{z/2}}\Theta \left(\pi n, \frac{z}{2}\right)} -R(z), 
\end{align}
where
\begin{align}\label{arz}
{R}(z) := {\pi ^{z/2}}\Gamma \left( {\frac{{-z}}{2}} \right)\zeta (-z)Z\left(1 + \frac{z}{2},\frac{z}{2}\right) + {\pi ^{-z/2}}\Gamma \left( {\frac{{z}}{2}} \right)\zeta (z)Z\left(1-\frac{z}{2},\frac{z}{2}\right).
\end{align}
\end{theorem}
With $\alpha\beta=1$, and the pair $\left(\varphi(x,z),\psi(x,z)\right)=\left(K_{z}(2\alpha x),\beta K_{z}(2\beta x)\right)$ which easily satisfies \eqref{recip1} and \eqref{recip2} (as can be seen from \eqref{koshlyakov-1}, we obtain the following result \cite[Equation (3.18)]{transf}:
\begin{corollary}\label{ramgenecor}
Let $-1<$ \textup{Re}$(z)<1$. Then
\begin{align}\label{mainagain3}
&-\frac{32}{\pi}\int_{0}^{\infty}\Xi\left(\frac{t+iz}{2}\right)\Xi\left(\frac{t-iz}{2}\right)\frac{\cos\left(\frac{1}{2}t\log\alpha\right)}{(t^2+(z+1)^2)(t^2+(z-1)^2)}\, dt\nonumber\\
&=\sqrt{\alpha}\left(\alpha^{\frac{z}{2}-1}\pi^{\frac{-z}{2}}\Gamma\left(\frac{z}{2}\right)\zeta(z)+\alpha^{-\frac{z}{2}-1}\pi^{\frac{z}{2}}\Gamma\left(\frac{-z}{2}\right)\zeta(-z)-4\sum_{n=1}^{\infty}\sigma_{-z}(n)n^{z/2}K_{\frac{z}{2}}\left(2n\pi\alpha\right)\right).
\end{align}
\end{corollary}
This result is illustrated below.
\begin{figure}[H]
\includegraphics[scale=0.835]{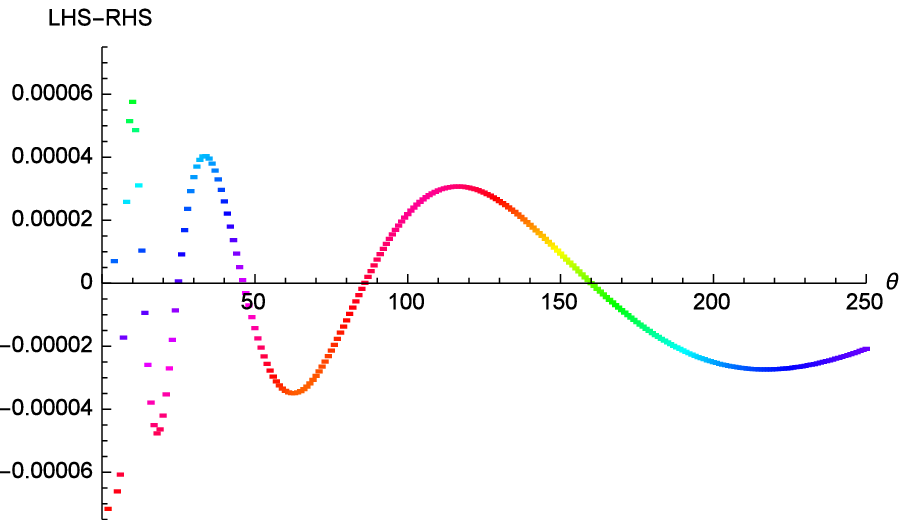}
\includegraphics[scale=0.835]{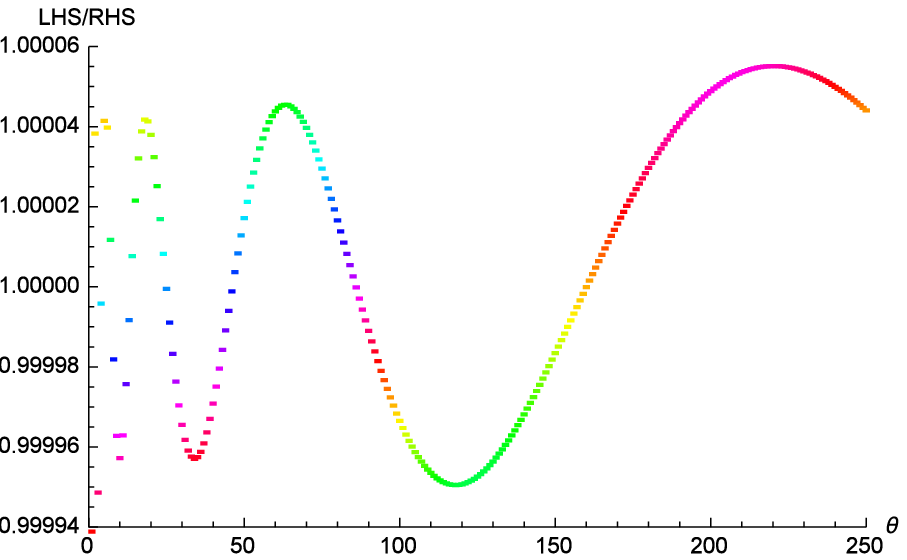}
\caption{\underline{Left}: Difference between the left and right sides of \eqref{mainagain3}; \newline \underline{Right}: Quotient $\frac{\text{left side}}{\text{right side}}$ of \eqref{mainagain3};\newline
Here $z=0$ and the series on the right is truncated to $10$ sums.}
\end{figure}
This identity is connected with the Ramanujan-Guinand formula (See Equation \eqref{mainagain} below.).
Our second result is
\begin{theorem}\label{hurgene}
Let $\eta>1/4$ and $0<\omega\leq \pi$. Suppose that $\varphi,\psi \in \Diamond_{\eta,\omega}$, are reciprocal in  the Koshliakov kernel as per \eqref{recip1} and \eqref{recip2}, and that $-1/2<$ \textup{Re}$(z)<1/2$. Let $Z(s, z)$ and $\Theta(x, z)$ be defined in \eqref{add}. Then,
\begin{align}\label{hurgeneid}
  &\pi^{\frac{z-3}{2}}\int_0^\infty  \Gamma \left( {\frac{{ z- 1 + it}}{4}} \right)\Gamma \left( {\frac{{z - 1 - it }}{4}} \right)\Xi \left( {\frac{{t + iz}}{2}} \right)\Xi \left( {\frac{{t - iz}}{2}} \right)Z\left( {\frac{{1 + it}}{2}},\frac{z}{2} \right)\frac{dt}{{{{t^2} + {{(z+1)}^2}}}}  \nonumber \\
  &= {\pi ^{z+1/2}}\Gamma \left( {\frac{{z + 3}}{2}} \right)\sum\limits_{n = 1}^\infty  {{\sigma _{ - z}}(n){n^{z + 1}}\int_0^\infty  {\Theta \left(x,\frac{z}{2}\right)\frac{{{x^{1 + z/2}}}}{{{{({x^2} + {\pi ^2}{n^2})}^{(z + 3)/2}}}}\, dx} } -S(z),
	\end{align}
	where
	\begin{equation*}
  S(z):=  {2^{-1 - z}}\Gamma (1 + z)\zeta (1 + z)Z\left( {1 + \frac{z}{2}},\frac{z}{2} \right) + 2^{-z}\G(z)\zeta(z)Z\left( {1 - \frac{z}{2}},\frac{z}{2} \right).
\end{equation*}
\end{theorem}
As a special case when again $\alpha\beta=1$, and $\left(\phi(x,z),\psi(x,z)\right)=\left(K_{z}(2\alpha x),\beta K_{z}(2\beta x)\right)$, we obtain the following result established in \cite[Equation (4.23)]{transf}.
\begin{corollary}\label{hurr}
Let $-1<$ \textup{Re}$(z)<1$ and define
\begin{equation}\label{dvarphi}
\lambda(x, z)=\zeta(z+1,x)-\frac{x^{-z}}{z}-\frac{1}{2}x^{-z-1},
\end{equation}
where $\zeta(z,x)$ is the Hurwitz zeta function. Then
\begin{align}\label{mainn}
&\frac{8(4\pi)^{\frac{z-3}{2}}}{\Gamma(z+1)}\int_{0}^{\infty}\Gamma\left(\frac{z-1+it}{4}\right)\Gamma\left(\frac{z-1-it}{4}\right)
\Xi\left(\frac{t+iz}{2}\right)\Xi\left(\frac{t-iz}{2}\right)\frac{\cos\left( \tf{1}{2}t\log\a\right)}{t^2+(z+1)^2}\, dt\nonumber\\
&=\a^{\frac{z+1}{2}}\left(\sum_{n=1}^{\infty}\lambda(n\a, z)-\frac{\zeta(z+1)}{2\alpha^{z+1}}-\frac{\zeta(z)}{\alpha z}\right).
\end{align}
\end{corollary}
As in the previous case, the graphical illustration of this corollary is given below.
\begin{figure}[H]
\includegraphics[scale=0.835]{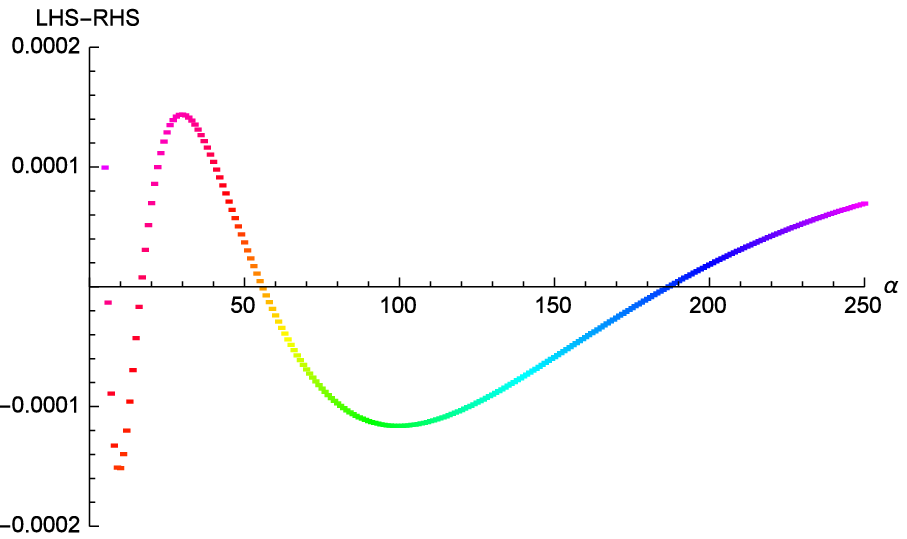}
\includegraphics[scale=0.835]{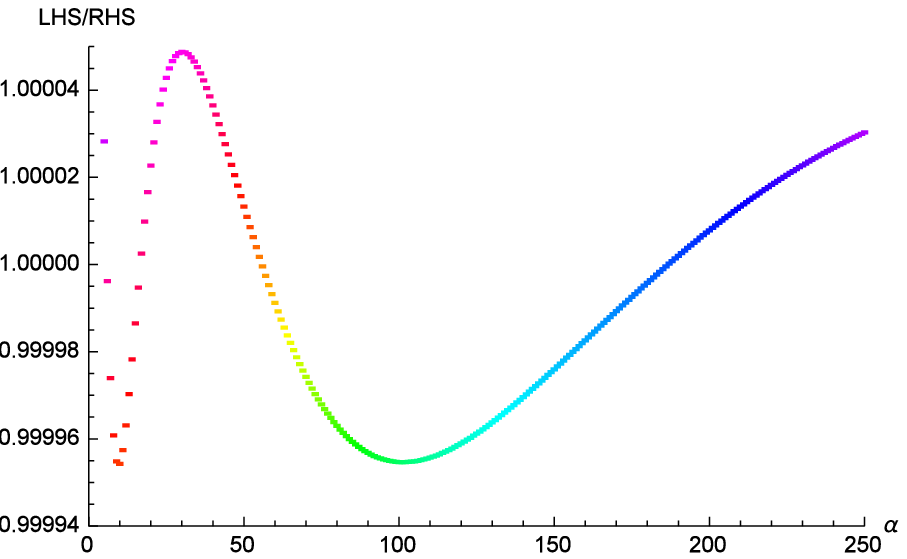}
\caption{\underline{Left}: Difference between the left and right sides of \eqref{mainn}; \newline \underline{Right}: Quotient $\frac{\text{left side}}{\text{right side}}$ of \eqref{mainn};\newline
Here $z=3/4$ and the series on the right is truncated to $10$ sums.}
\end{figure}
This is related to the modular-type transformation involving infinite series of Hurwitz zeta function. (See Equation \eqref{mainneq22} below.)

Theorem 5.3 from \cite{dixitmoll} gives a sufficient condition for a function to be equal to its first Koshliakov transform. In the same paper \cite[Equations (4.8), (4.17)]{dixitmoll}, there are two new explicit examples of such a function whereas \cite[p.~72, Equation (15.18)]{bdrz} contains a further new one. One may be able to obtain further corollaries of our theorems by working with these other examples. However, none of them is as simple as $K_{z}(x)$, and so we do not pursue this matter here. We note that Theorems \ref{ramguigene} and \ref{hurgene} work for any pair of functions reciprocal in the kernel \eqref{kernel}, not necessarily self-reciprocal. Dixon and Ferrar \cite[Section 5]{dixfer3} study, for example, the pair $\left(\varphi(x), \psi(x)\right)=\left(e^{-x},-\frac{2}{\pi}\left(e^{4x}\text{li}(e^{-4x})+e^{-4x}\text{li}(e^{4x})\right)\right)$, which is reciprocal in \eqref{kernel} with $z=0$. Here $\text{li}(x)$ denotes the logarithmic integral function defined by the Cauchy principal value
\begin{align*}
\operatorname{li} (x) = \mathop {\lim }\limits_{\varepsilon  \to {0^ + }} \left( {\int_0^{1 - \varepsilon } {\frac{{dt}}{{\log t}}}  + \int_{1 + \varepsilon }^x {\frac{{dt}}{{\log t}}} } \right).
\end{align*}

Finally, it must be mentioned here that Guinand \cite[Theorem 6]{gui}, \cite[Equation (1)]{guinand} derived 
the following summation formula involving $\sigma_{s}(n)$:
\begin{align}\label{guisum}
&\sum_{n=1}^{\infty}\sigma_{-z}(n)n^{\frac{z}{2}}f(n)-\zeta(1+z)\int_{0}^{\infty}x^{\frac{z}{2}}f(x)\, dx-\zeta(1-z)\int_{0}^{\infty}x^{-\frac{z}{2}}f(x)\, dx\nonumber\\
&=\sum_{n=1}^{\infty}\sigma_{-z}(n)n^{\frac{z}{2}}g(n)-\zeta(1+z)\int_{0}^{\infty}x^{\frac{z}{2}}g(x)\, dx-\zeta(1-z)\int_{0}^{\infty}x^{-\frac{z}{2}}g(x)\, dx.
\end{align}
Here $f(x)$ satisfies certain appropriate conditions (see \cite{gui} for details) and $g(x)$ is the 
transform of $f(x)$ with respect to
\begin{equation*}
-2\pi\sin\left(\tfrac{1}{2}\pi z\right)J_{z}(4\pi\sqrt{x})-\cos\left(\tfrac{1}{2}\pi z\right)\left(2\pi Y_{z}(4\pi\sqrt{x})-4K_{z}(4\pi\sqrt{x})\right),
\end{equation*}
which, up to a constant factor, is nothing but the Koshliakov kernel in \eqref{kernel}. Nasim \cite{nasim-1969a} derived a transformation formula involving functions reciprocal in Koshliakov kernel, and which is similar to \eqref{guisum}.
\section{Preliminaries}\label{prelim}
The Riemann zeta function satisfies the following functional equation \cite[p.~22, eqn. (2.6.4)]{titch}
\begin{equation}\label{zetafe}
\pi^{-\frac{s}{2}}\Gamma\left(\frac{s}{2}\right)\zeta(s)=\pi^{-\frac{(1-s)}{2}}\Gamma\left(\frac{1-s}{2}\right)\zeta(1-s),
\end{equation}
sometimes also written in the form 
\begin{equation}\label{zetaalt}
\xi(s)=\xi(1-s),
\end{equation}
where $\xi(s)$ is defined in \eqref{xi}.

For Re$(s)>\max(1,1+$Re$(a))$, the following Dirichlet series representation is well known \cite[p.~8, Equation (1.3.1)]{titch}:
	\begin{equation}\label{doublezeta}
\zeta(s)\zeta(s-a) = \sum\limits_{n = 1}^\infty  \frac{{\sigma _{a}(n)}}{n^{s}},
\end{equation}
Throughout the paper we use $R_a$ to denote the residue of a function being considered at its pole $z=a$.

We also use Parseval's theorem in the form \cite[p.~83, Equation (3.1.13)]{kp}
\begin{equation}\label{par}
\frac{1}{{2\pi i}}\int_{(\sigma )} {\mathfrak{G}(s)\mathfrak{H}(s){t^{ - s}}ds}  = \int_0^\infty  {g(x)h\left( {\frac{t}{x}} \right)\frac{{dx}}{x}},
\end{equation}
where $\mathfrak{G}$ and $\mathfrak{H}$ are Mellin transforms of $g$ and $h$ respectively.

The following lemma will be instrumental in proving our theorems.
\begin{lemma} \label{FEandEstimate}
Let $\eta>0$ and $0<\omega\leq \pi$. Let $-1/4<\textup{Re}(z)<1/4$. Suppose that $\varphi, \psi \in \Diamond_{\eta, \omega}$ are Koshliakov reciprocal functions. One has 
\begin{enumerate}
\item $Z(s, z)=Z(1-s, z)$ for all $s$ in $-\eta-|\textup{Re}(z)|/2 < \textup{Re}(s) < 1 + \eta+|\textup{Re}(z)|/2$.
\item $Z(\sigma + it, z) \ll_z e^{\left(\tfrac{\pi}{2} - \omega + \varepsilon\right)|t|}$ for every $\varepsilon >0$.
\end{enumerate}
\end{lemma}
\begin{proof}
Fix $\eta>0$, $0 < \omega \le \pi$ and let $\varphi, \psi \in \Diamond_{\eta,\omega}$. We begin by proving the first part of the claim involving the functional equation between $Z_1$ and $Z_2$. Set
\begin{equation}\label{defh}
w(s,z) := \Gamma \left( {\frac{s-z}{2} } \right)\Gamma \left( {\frac{s+z}{2} } \right).
\end{equation}
Note that
\begin{align}\label{comp}
  &w(1 - s,z){Z_1}(1 - s,z) \nonumber\\
	&= \int_0^\infty  {{x^{ - s}}\varphi (x, z)\, dx}  \nonumber \\
	 &= 2\int_0^\infty  {{x^{ - s}}\int_0^\infty  {\psi (t,z)\left( {\cos \left( {\pi z} \right){M_{2z}}(4\sqrt {tx} ) - \sin \left( {\pi z} \right){J_{2z}}(4\sqrt {tx} )} \right)\, dt}\, dx}  \nonumber \\
   &= 2\int_0^\infty  {\psi (t,z)\int_0^\infty  {{x^{ - s}}\left( {\cos \left( {\pi z} \right){M_{2z}}(4\sqrt {tx} ) - \sin \left( {\pi z} \right){J_{2z}}(4\sqrt {tx} )} \right)\, dx}\, dt}  \nonumber \\
   &= 2\pi^{2-2s}\int_0^\infty  {{t^{s - 1}}\psi (t,z)\int_0^\infty  {{u^{ - s}}\left( {\cos \left( {\pi z} \right){M_{2z}}(4\pi\sqrt u ) - \sin \left( {\pi z} \right){J_{2z}}(4\pi\sqrt u )} \right)\, du}\, dt}  \nonumber \\
   &= \frac{{{2^{2s - 1}}}}{\pi }\Gamma \left( {1 - s - z} \right)\Gamma \left( {1 - s + z} \right)\left( {\cos \left( {\pi z} \right) - \cos (\pi s)} \right)\int_0^\infty  {{t^{s - 1}}\psi (t,z)\, dt}  \nonumber \\
   &= \frac{{{2^{2s - 1}}}}{\pi }\Gamma \left( {1 - s - z} \right)\Gamma \left( {1 - s + z} \right)\left( {\cos \left( {\pi z} \right) - \cos (\pi s)} \right)w(s,z){Z_2}(s,z), 
\end{align}
where in the penultimate step, we used the integral evaluation
\begin{align*}
&\int_{0}^{\infty} x^{-s} \left( \cos \left( \pi z \right) 
M_{2z}( 4 \pi \sqrt{x}) - \sin \left( \pi z \right) 
J_{2z}(4 \pi \sqrt{x}) \right)\, dx \\
&= \frac{1}{2^{2-2s}\pi^{3-2s} } 
\Gamma \left( 1-s - z \right)
\Gamma \left( 1-s + z \right)
\left( \cos \left( \pi z \right) - \cos(\pi s) \right),
\end{align*}
valid for $\frac{1}{4}<$ Re$(s)<1\pm$ Re$(z)$. This can in turn be obtained by replacing $s$ by $1-s$, $z$ by $2z$, and letting $y=1$ in Lemma 5.1 of \cite{dixitmoll}. Using the duplication formula for the gamma function
\begin{align}
\G(s)\G\left(s+\frac{1}{2}\right)=\frac{\sqrt{\pi}}{2^{2s-1}}\G(2s),\label{dup}
\end{align}
and the reflection formula
\begin{align*}
\G(s)\G(1-s)=\frac{\pi}{\sin(\pi s)},\label{ref}
\end{align*}
we see that
\begin{equation*}
Z_1(1-s,z) = Z_2(s,z).
\end{equation*}
Similarly, $Z_2(1-s,z) = Z_1(s,z)$, and hence from \eqref{add}, we see that 
\begin{equation*}
Z(1-s,z)= Z(s,z).
\end{equation*}
The interchange of the order of integration in the third step in \eqref{comp} requires justification. We provide that below using Fubini's theorem. We only show that the double integral 
\begin{equation*}
\int_0^\infty  {\psi (t,z)\int_0^\infty  {{x^{ - s}}\left( {\cos \left( {\pi z} \right){M_{2z}}(4\sqrt {tx} ) - \sin \left( {\pi z} \right){J_{2z}}(4\sqrt {tx} )} \right)\, dx}\, dt}
\end{equation*}
converges absolutely for $\tfrac{3}{4} < \textup{Re}(s) < 1-\left|\textup{Re}(z)\right|$. (This necessitates Re$(z)$ to be between $-1/4$ and $1/4$.) The absolute convergence of the other one can be proved similarly.

Fix $\varepsilon_0>0$ such that
\begin{equation*}
\frac{3}{4} + {\varepsilon _0} \leqslant \operatorname{Re} (s) \leqslant 1 -\left|\textup{Re}(z)\right|- {\varepsilon _0}.
\end{equation*}
Then
\begin{equation}\label{xgrowth}
|x^{ - s}| \le 
			  \begin{cases}
				x^{\varepsilon_0-1+|\textup{Re}(z)|}, &\mbox{ if } \quad 0 \le x \le 1, \\
        x^{-\varepsilon_0-3/4}, &\mbox{ if } \quad x \ge 1.			
				\end{cases}
\end{equation}
The asymptotics of Bessel functions of the first and second kinds \cite[p.~360, 364]{as} give
\begin{equation*}
J_z(v) \ll_z 
\begin{cases}
v^{\textup{Re}(z)}, &\mbox{ if } \quad  0 \le v \le 1, \\
v^{-1/2}, &\mbox{ if } \quad  v > 1,			
\end{cases}			
\end{equation*}
\begin{equation*}
Y_z(v) \ll_z 
				\begin{cases}
				1+|\log v|, &\mbox{ if } \quad z=0, 0 \le v \le 1, \\
				v^{-|\textup{Re}(z)|}, &\mbox{ if } \quad z\neq 0,  0 \le v \le 1, \\
        v^{-1/2}, &\mbox{ if } \quad v > 1,
				\end{cases}			
\end{equation*}
whereas those for the modified Bessel function \cite[p.~375, 378]{as} give
\begin{equation}\label{kasym}
K_z(v) \ll_z 
\begin{cases}
1+|\log v|, &\mbox{ if } \quad z=0, 0 \le v \le 1, \\
v^{-|\textup{Re}(z)|}, &\mbox{ if } \quad z\neq 0,  0 \le v \le 1, \\
v^{-1/2}e^{-v}, &\mbox{ if } \quad v > 1.		
\end{cases}			
\end{equation}
Therefore
\begin{equation}\label{kergrowth}
\left\lvert{\cos \left( {\pi z} \right){M_{2z}}(4\sqrt {tx} ) - \sin \left( {\pi z} \right){J_{2z}}(4\sqrt {tx} )}\right\rvert \ll_z 
				\begin{cases}
				1+|\log (tx)|, &\mbox{ if } \quad z=0, 0 \le tx \le 1, \\
				(tx)^{-|\textup{Re}(z)|}, &\mbox{ if } \quad z\neq 0,  0 \le tx \le 1, \\
				(tx)^{-1/4}, &\mbox{ if } \quad tx \ge 1.			
				\end{cases}	
\end{equation}
We now divide the first quadrant of the $t,x$ plane into six different regions whose boundaries are determined by the $t=1$, $x=1$ and the hyperbola $xt=1$. 
\begin{figure}[h]
		\includegraphics[scale=0.96]{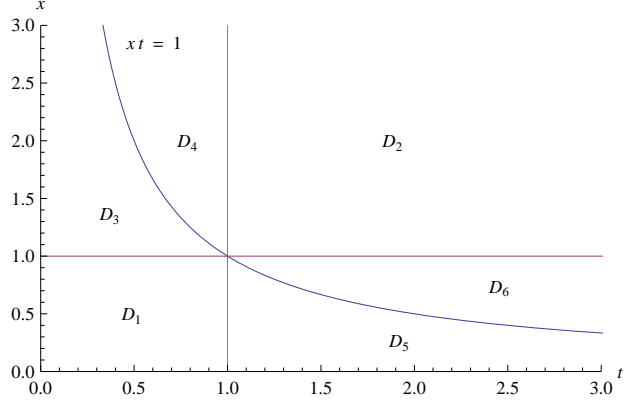}
			\caption{Regions from the hyperbola $xt=1$.}
\end{figure}
\newline Let $D=[0,\infty)\times[0,\infty)$. Set
\begin{equation*}
F_z(t,x): = x^{ - s}\psi (t, z)\left({\cos \left( {\pi z} \right){M_{2z}}(4\sqrt {tx} ) - \sin \left( {\pi z} \right){J_{2z}}(4\sqrt {tx} )}\right)
\end{equation*}
and
\begin{equation*}
I_{\psi}(s,z):=\int_{D}|F_z(x,t)|d\lambda,
\end{equation*}
where $d\lambda$ denotes the Lebesgue measure. Let
\begin{equation*}
I_{\psi}(s,z): = I_1(s,z) + I_2(s,z) +  \cdots + I_6(s,z), \quad \textnormal{with} \quad I_j(s,z): = \iint_{D_j} |F_z(x,t)|d\lambda,
\end{equation*}
 We estimate each $I_j$ separately. Let us assume that $\varepsilon < \varepsilon_0$. Using \eqref{growth}, \eqref{xgrowth}, and \eqref{kergrowth} in the regions $D_1$,$D_2$,$\cdots$,$D_6$, we have
\begin{align*}
  I_1(s,z) &
 \ll_{\varepsilon, z} \iint_{{D_1}} {\frac{1}{{{x^{1 - {\varepsilon _{_0}}-|\textup{Re}(z)|}}}}\frac{1}{{{t^{\delta}}}}\frac{1}{{{{(xt)}^{\varepsilon+|\textup{Re}(z)|} }}}\, dx\, dt}= {\int_0^1 {\frac{{dx}}{{{x^{1 - {\varepsilon _0} + \varepsilon }}}}} }  {\int_0^1 {\frac{{dt}}{{{t^{\delta+ \varepsilon+|\textup{Re}(z)| }}}}} } < \infty,  \nonumber  \\
{I_2(s,z)}& \ll \iint_{{D_2}} {\frac{1}{{{x^{3/4 + {\varepsilon _{_0}}}}}}\frac{1}{{{t^{1 + \eta+|\textup{Re}(z)| }}}}\frac{1}{{{t^{1/4}}{x^{1/4}}}}\, dx\, dt}\ll  {\int_1^\infty  {\frac{{dx}}{{{x^{1 + {\varepsilon _0}}}}}} }  {\int_1^\infty  {\frac{{dt}}{{{t^{5/4 + \eta }}}}} } < \infty ,\\
  {I_3(s,z)} &\ll \iint_{{D_3}} {\frac{1}{{{x^{3/4 + {\varepsilon _{_0}}}}}}\frac{1}{{{t^{\delta}}}}\frac{1}{{{{(xt)}^{\varepsilon+|\textup{Re}(z)|} }}}\, dx\, dt}\nonumber\\
	&\ll_{\varepsilon} \iint_{{D_3}} {\frac{1}{{{x^{3/4 + {\varepsilon _{_0}} + \varepsilon +|\textup{Re}(z)|}}}}\frac{1}{{{t^{\delta + \varepsilon+|\textup{Re}(z)| }}}}\, dx\, dt} \nonumber \\
   &= \int_1^\infty  {\frac{1}{{{x^{3/4 + {\varepsilon _0} + \varepsilon+|\textup{Re}(z)| }}}}{\int_0^{1/x} {\frac{{dt}}{{{t^{\delta + \varepsilon+|\textup{Re}(z)| }}}}} }\, dx} \\
   &\quad\ll \int_1^\infty  {\frac{{dx}}{{{x^{7/4 - \delta + {\varepsilon _0}}}}}} < \infty, \nonumber \\
{I_4(s,z)} &\ll \iint_{{D_4}} {\frac{1}{{{x^{3/4 + {\varepsilon _{_0}}}}}}\frac{1}{{{t^{\delta}}}}\frac{1}{{{t^{1/4}}{x^{1/4}}}}\, dx\, dt} = \int_1^\infty  {\frac{1}{{{x^{1 + {\varepsilon _0}}}}}\int_{1/x}^1 {\frac{{dt}}{t^{\delta+1/4}}} dx} \ll \int_1^\infty  {\frac{{1}}{{{x^{1 + {\varepsilon _0}}}}}dx}  < \infty, \\
  I_5(s,z) &\ll \iint_{{D_5}} {\frac{1}{{{x^{1-|\textup{Re}(z)| - {\varepsilon _{_0}}}}}}\frac{1}{{{t^{1 +|\textup{Re}(z)|+ \eta }}}}\frac{1}{{{{(xt)}^{\varepsilon+|\textup{Re}(z)|} }}}\, dx\, dt}\nonumber \\
	&= \iint_{{D_5}} {\frac{1}{{{x^{1 - {\varepsilon _{_0}} + \varepsilon }}}}\frac{1}{{{t^{1+2|\textup{Re}(z)| + \eta  + \varepsilon }}}}\, dx\, dt} \nonumber \\
   &= \int_1^\infty  {\frac{1}{{{t^{1 +2|\textup{Re}(z)|+ \eta  + \varepsilon }}}}\bigg( {\int_0^{1/t} {\frac{{dx}}{{{x^{1 - {\varepsilon _0} + \varepsilon }}}}} } \bigg)dt}\\ 
   &\ll _{\varepsilon ,\varepsilon _0}\int_1^\infty  {\frac{{dt}}{{{t^{1 +2|\textup{Re}(z)|+ \eta  + {\varepsilon _0}}}}}}  < \infty . \nonumber 
\end{align*}
Finally,
\begin{align*}
{I_6(s,z)} &\ll \iint_{{D_6}} \frac{1}{{{x^{1 -|\textup{Re}(z)|- {\varepsilon _{_0}}}}}}\frac{1}{{{t^{1 +|\textup{Re}(z)|+ \eta }}}}\frac{1}{{{t^{1/4}}{x^{1/4}}}}\, dx\, dt\\\nonumber &= \int_1^\infty  {\frac{1}{{{t^{5/4+|\textup{Re}(z)| + \eta }}}}\int_{1/t}^1 {\frac{{dx}}{{{x^{5/4 -|\textup{Re}(z)|- {\varepsilon _0}}}}}} dt}\\\nonumber
  &\ll \int_1^\infty  {\frac{1}{{{t^{1 +2|\textup{Re}(z)|+ \eta  + {\varepsilon _0}}}}}dt}  < \infty .\end{align*}
Hence $I_{\psi}(s,z)<\infty$, which justifies the interchange of the order of integration in \eqref{comp}. Finally, let us look at $Z_2(s,z)$ and the Mellin transform of $\psi(t,z)$. We split up the integral as
\begin{equation*}
\int_0^\infty  {|\psi (t,z)||{t^{s - 1}}|\, dt}  = \int_0^1 {|\psi (t,z)||{t^{s - 1}}|\, dt}  + \int_1^\infty  {|\psi (t,z)||{t^{s - 1}}|\, dt} .
\end{equation*}
Since $\psi\in \Diamond_{\eta, \omega}$, for the latter integral, we have
\begin{equation*}\int_1^\infty  {|\psi (t,z)||{t^{s - 1}}|\, dt}  \ll \int_1^\infty  {{t^{\operatorname{Re} (s) - 2-|\textup{Re}(z)|/2 - \eta }}\, dt}  
< \infty, \end{equation*}
provided that $\textup{Re}(s)<1+\eta+\tfrac{|\textup{Re}(z)|}{2}$. 
Similarly, for the first integral, we have
\begin{equation*}
\int_0^1 {|\psi (t, z)|{t^{\operatorname{Re} (s) - 1}}dt} \ll \int_0^1 {{t^{\operatorname{Re} (s) - 1 -\delta}}}\, dt < \infty,
\end{equation*}
provided that $\textup{Re}(s) > \delta$.
This shows that the function $Z_2(s,z)$ is well-defined and analytic, as a function of $s$, in the region
\begin{align} \label{etarange}
\delta < \textup{Re}(s) < 1 + \eta+\frac{|\textup{Re}(z)|}{2},
\end{align}
for every $\delta>0$. Similarly, $Z_1(s,z)$ is well-defined and analytic in $s$ in this region. Thus, by analytic continuation, the equality $Z_1(1-s,z)=Z_2(s,z)$ holds in the vertical strip
\begin{equation*}
-\eta-\frac{|\textup{Re}(z)|}{2} < \textup{Re}(s) < 1 + \eta+\frac{|\textup{Re}(z)|}{2}.
\end{equation*}
%
To prove the second part of the lemma, let us consider the line along any radius vector $r$ and angle $\theta$ (we would choose $-\theta$, if $t=$ Im$(s)<0$), where $|\theta|<\omega$. Then by Cauchy's theorem we can deform the integral \eqref{defZ1} to 
\begin{equation*}
h(\sigma+it,z) Z_1(\sigma+it,z)=\int_{0}^{\infty}r^{\sigma+it-1}e^{i\theta(\sigma+it)}\varphi(re^{i\theta},z)\, dr,
\end{equation*}
where $\theta, t>0$. Therefore, by splitting the range of integration to $[0,1]$ and $[1,\infty]$ and the fact that $Z_1$ is analytic in the region defined by \eqref{etarange} we see that
\begin{equation}\label{absz}
 \left \lvert h(\sigma+it,z) Z_1(\sigma+it,z) \right \rvert 
\leq e^{-\theta t}\int_{0}^{\infty}r^{\sigma-1}|\varphi(re^{i\theta}, z)|\, dr \ll e^{-|\theta||t|},
\end{equation}
since $\varphi \in \Diamond_{\eta,\omega}$. By Stirling's formula for $\Gamma(\sigma+it)$ in the vertical strip $p\leq\sigma\leq q$ \cite[p.~224]{cop}, we have, as $|t|\to\infty$,
\begin{equation}\label{strivert}
  |\Gamma(s)|=\sqrt{2\pi}|t|^{\sigma-\frac{1}{2}}e^{-\frac{1}{2}\pi |t|}\left(1+O\left(\frac{1}{|t|}\right)\right),
\end{equation}
as $|t|\to \infty$. Now combining \eqref{defh}, \eqref{absz} and \eqref{strivert} we get 
\begin{equation*}
  Z_1(1-s,z)=Z_2(s,z) \ll_z  e^{(\frac{\pi}{2}-\omega+\varepsilon)|t|},
  \end{equation*}
for every $\varepsilon>0$. This completes the proof the lemma.
\end{proof}
\section{Generalization of an integral identity connected with the Ramanujan-Guinand formula}\label{1.2}
We prove Theorem \ref{ramguigene} here. The convergence of the series in \eqref{ramguigeneid} follows from the fact that $\Theta(s,z)\in\Diamond_{\eta,\omega}$.

Using \eqref{xif} and \eqref{zetaalt}, it is easy to see that
\begin{equation*}
\Xi \left( {\frac{{t + iz}}{2}} \right)\Xi \left( {\frac{{t - iz}}{2}} \right) = \Xi \left( {\frac{{ - t + iz}}{2}} \right)\Xi \left( {\frac{{ - t - iz}}{2}} \right).
\end{equation*}
Along with \eqref{fdefi} and part (i) of Lemma \ref{FEandEstimate}, this gives
\begin{align}\label{imp}
  &\int_0^\infty  {f\left( {\frac{t}{2}, z} \right)\Xi \left( {\frac{{t + iz}}{2}} \right)\Xi \left( {\frac{{t - iz}}{2}} \right)Z\left( {\frac{{1 + it}}{2}, \frac{z}{2}} \right)\, dt}  \nonumber \\
   &= \frac{1}{2}\int_{ - \infty }^\infty  {f\left( {\frac{t}{2}, z} \right)\Xi \left( {\frac{{t + iz}}{2}} \right)\Xi \left( {\frac{{t - iz}}{2}} \right)Z\left( {\frac{{1 + it}}{2},\frac{z}{2}} \right)\, dt}  \nonumber \\
   &= \frac{1}{i}\int_{(\tfrac{1}{2})} {\phi\left( {s - \frac{1}{2}},z \right)\phi\left( {\frac{1}{2} - s},z \right)\xi \left( {s - \frac{z}{2}} \right)\xi \left( {s + \frac{z}{2}} \right)Z\left(s,\frac{z}{2}\right)\, ds}.
\end{align}
Now choose
\begin{equation*}
\phi(s, z) := {\left( {\left( {s + \frac{1}{2} + \frac{z}{2}} \right)\left( {s + \frac{1}{2} - \frac{z}{2}} \right)} \right)^{ - 1}}
\end{equation*}
so that from \eqref{imp},
\begin{align}\label{firim}
&16\int_0^\infty  {\Xi \left( {\frac{{t + iz}}{2}} \right)\Xi \left( {\frac{{t - iz}}{2}} \right)Z\left( {\frac{{1 + it}}{2}}, \frac{z}{2} \right)\frac{{dt}}{{({t^2}+(z+1)^2)({t^2}+(z-1)^2)}}}  \nonumber \\
  &=\frac{1}{{4i}}\int_{(\tfrac{1}{2})} {{\pi ^{ - s}}\Gamma \left( {\frac{s}{2} - \frac{z}{4}} \right)\Gamma \left( {\frac{s}{2} + \frac{z}{4}} \right)\zeta \left( {s - \frac{z}{2}} \right)\zeta \left( {s + \frac{z}{2}} \right)Z\left(s,\frac{z}{2}\right)\, ds}
	\end{align}
	Since $-1/2 < $ Re$(z)<1/2$ and Re$(s)=1/2$, we have $1/4 < $ Re$(s-z/2)<3/4$. In order to use \eqref{doublezeta}, with $s$ replaced by $s-z/2$ and $a$ replaced by $-z$, which is therefore valid for Re$(s) > 1 \pm $ Re$\left(\frac{z}{2}\right)$, we shift the line of integration from Re$(s)=1/2$ to Re$(s)=5/4$. In doing so, we encounter a pole of order $1$ at $s=1-z/2$ (due to $\zeta(s+z/2)$) and a pole of order $1$ at $s=1+z/2$ (due to $\zeta(s-z/2)$). Using the notation for the residue of a function at a pole, we see that the integral in \eqref{firim} is equal to
\begin{align*}
   & \frac{1}{{4i}}\left(\int_{(\frac{5}{4})} {{\pi ^{ - s}}\Gamma \left( {\frac{s}{2} - \frac{z}{4}} \right)\Gamma \left( {\frac{s}{2} + \frac{z}{4}} \right)\zeta \left( {s - \frac{z}{2}} \right)\zeta \left( {s + \frac{z}{2}} \right)Z\left(s,\frac{z}{2}\right)\, ds}  -2\pi i\left(R_{1+\frac{z}{2}}+R_{1-\frac{z}{2}}\right)\right) \nonumber \\
   &= \frac{1}{{4i}}\bigg(\sum\limits_{n = 1}^\infty {\sigma _{ - z}}(n){{n^{z/2}}\int_{(\frac{5}{4} )} {\Gamma \left( {\frac{s}{2} - \frac{z}{4}} \right)\Gamma \left( {\frac{s}{2} + \frac{z}{4}} \right){{(\pi n)}^{ - s}}Z\left(s,\frac{z}{2}\right)\, ds} } \nonumber\\
	&\quad\quad-2\pi i\left({\pi ^{-\frac{(1 + z)}{2}}}\Gamma \left( {\frac{{1 + z}}{2}} \right)\zeta (1 + z)Z\left(1 + \frac{z}{2},\frac{z}{2}\right) +{\pi ^{-\frac{(1 - z)}{2}}}\Gamma \left( {\frac{{1 - z}}{2}} \right)\zeta (1 - z)Z\left(1 - \frac{z}{2},\frac{z}{2}\right)\right)\bigg)  \nonumber \\
   &= \frac{\pi }{2}\sum\limits_{n = 1}^\infty  {{\sigma _{ - z}}(n){{n^{z/2}}}\Theta\left(\pi n,\frac{z}{2}\right)} \nonumber\\
	&\quad- \frac{\pi }{2}\left({\pi ^{z/2}}\Gamma \left( {\frac{{-z}}{2}} \right)\zeta (-z)Z\left(1 + \frac{z}{2},\frac{z}{2}\right) + {\pi ^{-z/2}}\Gamma \left( {\frac{{z}}{2}} \right)\zeta (z)Z\left(1-\frac{z}{2},\frac{z}{2}\right)\right),
\end{align*}
where in the penultimate step, we used the part (ii)  of Lemma \ref{FEandEstimate} and \eqref{strivert} to see that the integrals along the horizontal segments of the contour $[\frac{1}{2}-iT, \frac{5}{4}-iT, \frac{5}{4}+iT, \frac{1}{2}+iT, \frac{1}{2}-iT]$ tend to zero as $T\to\infty$, then interchanged the order of summation and integration because of absolute convergence, and in the ultimate step we used \eqref{zetafe} twice.

This completes the proof of Theorem \ref{ramguigene}.

In the following corollary and at few other places in the sequel, we use the notation $\Theta(x)=\Theta(x,0)$ and $Z(s)=Z(s,0)$.
\begin{corollary}
Let $d(n)$ denote the number of positive divisors of $n$. Then
\begin{align*}
\frac{32}{\pi}\int_0^\infty  {{\Xi ^2}\left( {\frac{t}{2}} \right)Z\left( {\frac{{1 + it}}{2}} \right)\frac{{dt}}{{{{(1 + {t^2})}^2}}}}  = {\sum\limits_{n = 1}^\infty  {d(n)\Theta (\pi n)}  - (Z'(1) + (\gamma  - \log 4\pi )Z(1))}.
\end{align*}
\end{corollary}
\begin{proof}
Let $z\to 0$ in Theorem \ref{ramguigene}, and note that 
the expansion around $z=0$ of $R(z)$, defined in \eqref{arz}, is given by
\begin{equation*}
R(z) =  (\gamma  - \log 4\pi )Z(1) + Z'(1) + O(z),
\end{equation*}
where $\gamma$ is Euler's constant.
\end{proof}

\subsection{Proof of the integral identity connected with the Ramanujan-Guinand formula}\label{rg}
Corollary \ref{ramgenecor} is proved here. As mentioned before its statement, we substitute $\varphi(x,z)=K_{z}(2\alpha x)$ in \eqref{recip2}, where $\a>0$. The fact that $K_{z}(x)\in\Diamond_{\eta,\omega}$ for any $\eta>0$ is obvious from \eqref{kasym}. Let $\b=1/\a$. From \eqref{koshlyakov-1}, it is easy to see then that $\psi(x, z)=\beta K_{z}(2\beta x)$. For Re$(s) > \pm $ Re$(\nu)$ and $q>0$, we have \cite[p.~115, Equation (11.1)]{ober}
\begin{equation}\label{melbess}
\int_{0}^{\infty}x^{s-1}K_{\nu}(qx)\, dx=2^{s-2}q^{-s}\Gamma\left(\frac{s-\nu}{2}\right)\Gamma\left(\frac{s+\nu}{2}\right).
\end{equation}
Using \eqref{defZ1}, \eqref{defZ2}, \eqref{melbess}, and the fact that $\a\b=1$, we see that
\begin{equation}\label{z1spl}
Z_1(s,z)=\frac{1}{4}\a^{-s}\hspace{2mm}\text{and}\hspace{2mm}Z_2(s,z)=\frac{1}{4}\a^{s-1}
\end{equation}
so that
\begin{equation}
Z\left(\frac{1+it}{2},\frac{z}{2}\right)=\frac{1}{2\sqrt{\a}}\cos\left(\frac{1}{2}t\log\a\right)
\end{equation}
and
\begin{equation}\label{thspl}
\Theta\left(\pi n,\frac{z}{2}\right)=K_{z/2}(2\pi n\a)+\b K_{z/2}(2\pi n\b).
\end{equation}
Then Theorem \ref{ramguigene} gives
\begin{align*}
-\frac{32}{\pi}\int_{0}^{\infty}\Xi\left(\frac{t+iz}{2}\right)\Xi\left(\frac{t-iz}{2}\right)\frac{\cos\left(\frac{1}{2}t\log\alpha\right)\, dt}{(t^2+(z+1)^2)(t^2+(z-1)^2)}=\frac{1}{2}\left(\mathfrak{F}(\a, z)+\mathfrak{F}(\b, z)\right),
\end{align*}
where
\begin{align*}
\mathfrak{F}(\a, z):=\sqrt{\alpha}\left(\alpha^{\frac{z}{2}-1}\pi^{\frac{-z}{2}}\Gamma\left(\frac{z}{2}\right)\zeta(z)+\alpha^{-\frac{z}{2}-1}\pi^{\frac{z}{2}}\Gamma\left(\frac{-z}{2}\right)\zeta(-z)-4\sum_{n=1}^{\infty}\sigma_{-z}(n)n^{z/2}K_{\frac{z}{2}}\left(2n\pi\alpha\right)\right).
\end{align*}
Now the Ramanujan-Guinand formula \cite{guinand}, \cite[p.~253]{lnb} gives, for $ab=\pi^2$,
\begin{multline}\label{mainagain}
\sqrt{a}\sum_{n=1}^{\infty}\sigma_{-z}(n)n^{z/2}K_{z/2}(2na)
-\sqrt{b}\sum_{n=1}^{\infty}\sigma_{-z}(n)n^{z/2}K_{z/2}(2nb)\\
=\df{1}{4}\Gamma\left(\dfrac{z}{2}\right)\zeta(z)\{b^{(1-z)/2}-a^{(1-z)/2}\}
+\df{1}{4}\Gamma\left(-\dfrac{z}{2}\right)\zeta(-z)\{a^{(1+z)/2}-a^{(1+z)/2}\}.
\end{multline}
(See \cite{bls} for history and other details.)

Invoking \eqref{mainagain}, with $a=\pi\a$ and $b=\pi\b$, it is seen that $\mathfrak{F}(\a, z)=\mathfrak{F}(\b, z)$, with the help of which we deduce that, for $-1/2<$ Re$(z)<1/2$,
\begin{align*}
-\frac{32}{\pi}\int_{0}^{\infty}\Xi\left(\frac{t+iz}{2}\right)\Xi\left(\frac{t-iz}{2}\right)\frac{\cos\left(\frac{1}{2}t\log\alpha\right)\, dt}{(t^2+(z+1)^2)(t^2+(z-1)^2)}
=\mathfrak{F}(\a, z).
\end{align*}
Since both sides of the above identity are analytic for $-1<$ Re$(z)<1$, by the principle of analytic continuation, the result holds for $-1<$ Re$(z)<1$ as well. This proves Corollary \ref{ramgenecor}.

\section{Generalization of an integral identity involving infinite series of Hurwitz zeta function}\label{1.3}
Theorem \ref{hurgene} is proved here. To see that the series on the right-hand side of \eqref{hurgeneid} is convergent, it suffices to show that 
\begin{equation*}
\mathfrak{P}(z):=\sum\limits_{n = 1}^\infty  {{\sigma _{ - z}}(n){n^{z + 1}}\int_0^1  {\Theta \left(x,\frac{z}{2}\right)\frac{{{x^{1 + z/2}}}}{{{{({x^2} + {\pi ^2}{n^2})}^{(z + 3)/2}}}}\, dx} }
\end{equation*}
converges. Since $\Theta\in\Diamond_{\eta,\omega}$, $\Theta\left(x,\frac{z}{2}\right)\ll x^{-\delta}$ for $0<x<1$ and for every $\delta>0$, so that the inside integral does not blow up as $x\to 0$. Thus,
\begin{align}
\mathfrak{P}(z)\ll\sum_{n=1}^{\infty}\frac{\sigma_{-\text{Re}(z)}(n)}{n^{2}}.
\end{align}
The series on the right-hand side converges as long as Re$(z) > -1$ as can be seen from \eqref{doublezeta}. Let 
\begin{equation*}
\phi(s, z)=\frac{1}{s+(z+1)/2}\G\left(\frac{z-1}{4}+\frac{s}{2}\right).
\end{equation*}
Using \eqref{fdefi} and \eqref{imp}, we see that
\begin{align}\label{adone}
&4\int_0^\infty  \Gamma \left( {\frac{{ z- 1 + it}}{4}} \right)\Gamma \left( {\frac{{z - 1 - it }}{4}} \right)\Xi \left( {\frac{{t + iz}}{2}} \right)\Xi \left( {\frac{{t - iz}}{2}} \right)Z\left( {\frac{{1 + it}}{2}},\frac{z}{2} \right)\frac{dt}{{{{t^2} + {{(z+1)}^2}}}}  \nonumber \\
&=  - \frac{1}{4i}\int_{\left(\tfrac{1}{2}\right)} \left(s - \frac{z}{2}\right)\left(s - 1 + \frac{z}{2}\right){\Gamma \left( \frac{s}{2} + \frac{z}{4} - \frac{1}{2} \right)\Gamma \left( \frac{z}{4} - \frac{s}{2} \right)}\Gamma \left( {\frac{s}{2} - \frac{z}{4}} \right)\Gamma \left( {\frac{s}{2} + \frac{z}{4}} \right)  \nonumber \\
   &\quad\quad\quad\quad\quad\quad\times\zeta \left( {s - \frac{z}{2}} \right)\zeta \left( {s + \frac{z}{2}} \right)Z\left(s, \frac{z}{2}\right){\pi ^{ - s}}\, ds \nonumber \\
  &=  \frac{1}{i}\bigg(\int_{\left(\tfrac{5}{4}\right)} {\Gamma \left( \frac{s}{2} + \frac{z}{4} + \frac{1}{2} \right)\Gamma \left( \frac{z}{4} - \frac{s}{2}+1 \right)}\Gamma \left( {\frac{s}{2} - \frac{z}{4}} \right)\Gamma \left( {\frac{s}{2} + \frac{z}{4}} \right)  \nonumber \\
   &\quad\quad\quad\quad\quad\times\zeta \left( {s - \frac{z}{2}} \right)\zeta \left( {s + \frac{z}{2}} \right)Z\left(s, \frac{z}{2}\right){\pi ^{ - s}}\, ds-2\pi i\left(R_{1+\frac{z}{2}}+R_{1-\frac{z}{2}}\right)\bigg)\nonumber\\
	&=  \frac{1}{i}\bigg(\sum_{n=1}^{\infty}\sigma_{-z}(n)n^{z/2}\int_{\left(\tfrac{5}{4}\right)} {\Gamma \left( \frac{s}{2} + \frac{z}{4} + \frac{1}{2} \right)\Gamma \left( \frac{z}{4} - \frac{s}{2}+1 \right)}\Gamma \left( {\frac{s}{2} - \frac{z}{4}} \right)\Gamma \left( {\frac{s}{2} + \frac{z}{4}} \right)  \nonumber \\
   &\quad\quad\quad\quad\quad\quad\quad\quad\quad\quad\quad\quad\times Z\left(s, \frac{z}{2}\right){(\pi n)^{ - s}}\, ds-2\pi i\left(R_{1+\frac{z}{2}}+R_{1-\frac{z}{2}}\right)\bigg),
\end{align}
where in the penultimate step we used the functional equation $\G(w+1)=w\G(w)$ of the Gamma function. Here $R_{1+\frac{z}{2}}$ is the residue of the integrand at the pole of order $1$ due to $\zeta\left(s-\frac{z}{2}\right)$, and $R_{1-\frac{z}{2}}$ is the residue of the integrand at the pole of order $1$ due to $\zeta\left(s-\frac{z}{2}\right)$. These residues turn out to be
\begin{align}\label{2res}
R_{1+\frac{z}{2}}&=2^{-z}\pi^{(1-z)/2}\G(1+z)\zeta(1+z)Z\left(1+\frac{z}{2},\frac{z}{2}\right),\nonumber\\
R_{1-\frac{z}{2}}&=2^{1-z}\pi^{(1-z)/2}\G(z)\zeta(z)Z\left(1-\frac{z}{2},\frac{z}{2}\right).
\end{align}
Note that for $0<$ Re $u<$ Re $w$, Euler's beta integral is given by
\begin{equation*}
\int_{0}^{\infty}\frac{x^{u-1}}{(1+x)^{w}}\, dx=\frac{\Gamma(u)\Gamma(w-u)}{\Gamma(w)},
\end{equation*}
so that for $-$ Re$\left(\frac{z}{2}\right)-1<d=$ Re$(s)<$ Re$\left(\frac{z}{2}\right)+2$,
\begin{equation*}
\frac{1}{2\pi i}\int_{(d)}\Gamma \left( \frac{s}{2} + \frac{z}{4} + \frac{1}{2} \right)\Gamma \left( \frac{z}{4} - \frac{s}{2}+1 \right)x^{-s}\, ds=2\G\left(\frac{z+3}{2}\right)\frac{x^{(z+2)/2}}{(1+x^2)^{(z+3)/2}}.
\end{equation*}
Along with \eqref{zthph} and \eqref{par}, this gives
\begin{align}\label{adone2}
&\int_{\left(\tfrac{5}{4}\right)} {\Gamma \left( \frac{s}{2} + \frac{z}{4} + \frac{1}{2} \right)\Gamma \left( \frac{z}{4} - \frac{s}{2}+1 \right)}\Gamma \left( {\frac{s}{2} - \frac{z}{4}} \right)\Gamma \left( {\frac{s}{2} + \frac{z}{4}} \right)Z\left(s, \frac{z}{2}\right){(\pi n)^{ - s}}\, ds\nonumber\\
&=4\pi i\G\left(\frac{z+3}{2}\right)(\pi n)^{(z+2)/2}\int_{0}^{\infty}{\Theta \left(x,\frac{z}{2}\right)\frac{{{x^{1 + z/2}}}}{{{{({x^2} + {\pi ^2}{n^2})}^{(z + 3)/2}}}}\, dx}.
\end{align}
From \eqref{adone}, \eqref{2res} and \eqref{adone2}, we obtain \eqref{hurgeneid}. This completes the proof of Theorem \ref{hurgene}.

\begin{corollary}
Let $d(n)=\sum_{d|n}1$ as before. Then,
\begin{align*}
   & \pi ^{-3/2}\int_0^\infty  {{{\left| {\Gamma \left( {\frac{{ - 1 + it}}{4}} \right)} \right|}^2}{\Xi ^2}\left( {\frac{t}{2}} \right)Z\left( {\frac{{1 + it}}{2}} \right)\frac{{dt}}{{1 + {t^2}}}}  \nonumber \\
   &= \frac{\pi}{2} \sum\limits_{n = 1}^\infty  {nd(n)\int_0^\infty  {\frac{x\Theta (x)}{{{{	({x^2} + {\pi ^2}{n^2})}^{3/2}}}}dx} }-\frac{1}{2}\left((\gamma  - \log (2\pi ))Z(1) + Z'(1)\right).  \nonumber 
\end{align*} 
\end{corollary}
\begin{proof}
Let $z\to 0$ in Theorem \ref{hurgene}, and note that 
\begin{equation*}
S(0)=\frac{1}{2}\left((\gamma  - \log (2\pi ))Z(1) + Z'(1)\right).
\end{equation*}
\end{proof}

\subsection{Proof of the integral identity involving infinite series of Hurwitz zeta function}\label{hur}
We now prove Corollary \ref{hurr}. We again choose the pair $\left(\phi(x,z),\psi(x,z)\right)=\left(K_{z}(2\alpha x),\beta K_{z}(2\beta x)\right)$ where $\a\b=1$. So from \eqref{z1spl}-\eqref{thspl},
\begin{align}\label{dun1}
&\frac{\pi^{\frac{z-3}{2}}}{2\sqrt{\a}}\int_{0}^{\infty}\Gamma\left(\frac{z-1+it}{4}\right)\Gamma\left(\frac{z-1-it}{4}\right)
\Xi\left(\frac{t+iz}{2}\right)\Xi\left(\frac{t-iz}{2}\right)\frac{\cos\left( \tf{1}{2}t\log\a\right)}{t^2+(z+1)^2}\, dt\nonumber\\
&={\pi ^{z+1/2}}\Gamma \left( {\frac{{z + 3}}{2}} \right)\sum\limits_{n = 1}^\infty  {{\sigma _{ - z}}(n){n^{z + 1}}\int_0^\infty  {\frac{{{x^{1 + z/2}}}}{{{{({x^2} + {\pi ^2}{n^2})}^{(z + 3)/2}}}}\left(K_{z/2}(2\a x)+\b K_{z/2}(2\b x)\right)\, dx} }\nonumber\\
&\quad-\left({2^{-3 - z}}\Gamma (1 + z)\zeta (1 + z)\left(\a^{-1 - \frac{z}{2}}+\a^{\frac{z}{2}}\right) + 2^{-2-z}\G(z)\zeta(z)\left(\a^{-1 + \frac{z}{2}}+\a^{-\frac{z}{2}}\right)\right).   
\end{align}
For Re$(a)>0$, Re$(b)>0$ and Re$(\nu)>-1$, we have \cite[p.~678, formula \textbf{6.565.7}]{gr}
\begin{align*}
\int_0^\infty  {{x^{1 + \nu}}{{({x^2} + a^2)}^{\mu}}{K_{\nu}}(bx)\, dx} = {2^{\nu}}\Gamma (\nu + 1){a^{\nu + \mu + 1}}{b^{ - 1 - \mu}}{S_{\mu - \nu,\mu + \nu + 1}}(ab),
	\end{align*}
whereas from the footnote on the first page of \cite{gl}, we have
\begin{equation*}
S_{\mu,\nu}(w)=w^{\mu+1}\int_{0}^{\infty}te^{-wt}{}_2F_{1}\left(\frac{1-\mu+\nu}{2},\frac{1-\mu-\nu}{2};\frac{3}{2};-t^2\right)\, dt.
\end{equation*}
Thus
\begin{equation}\label{lom2}
\int_0^\infty  {{x^{1 + \nu}}{{({x^2} + a^2)}^{\mu}}{K_{\nu}}(bx)\, dx}={a^{2\mu + 2}}{\left( {\frac{2}{b}} \right)^{\nu}}\Gamma (\nu + 1)\int_0^\infty  {y{e^{ - aby}}_2{F_1}\left( {1 + \nu, - \mu,\frac{3}{2}; - {y^2}} \right)\, dy}. 
\end{equation}
Let $a = \pi n$, $b = 2 \a$, $u=-(3+z)/2$ and $v = z/2$ in \eqref{lom2} so that
\begin{align*}
  \int_0^\infty  {\frac{{{x^{1 + z/2}}{K_{z/2}}(2\a x)}}{{{{({x^2} + {\pi ^2}{n^2})}^{(3 + z)/2}}}}\, dx}  & = \frac{{{{(\pi n)}^{ - 1 - z}}{\a ^{ - z/2}}}}{{1 + z}}\Gamma \left( 1+{\frac{z}{2}} \right)\nonumber\\
	&\quad\quad\times\int_0^\infty  {\frac{{{e^{ - 2\pi n\a y}}}}{{{{(1 + {y^2})}^{(1 + z)/2}}}}\sin ((1 + z)\arctan y)\, dy},
\end{align*}
since \cite[p.~1006, formula \textbf{9.121.4}]{gr}
\begin{equation*}
{}_2F_{1}\left(\frac{1-a}{2},\frac{2-a}{2};\frac{3}{2};\frac{w^2}{u^2}\right)=\frac{(u+w)^a-(u-w)^a}{2awu^{a-1}}.
\end{equation*}
Thus
\begin{align}\label{dun2.5}
&{\pi ^{z+1/2}}\Gamma \left( {\frac{{z + 3}}{2}} \right)\sum\limits_{n = 1}^\infty  {\sigma _{ - z}}(n){n^{z + 1}}\int_0^\infty  \frac{{{x^{1 + z/2}}}}{{{{({x^2} + {\pi ^2}{n^2})}^{(z + 3)/2}}}}K_{z/2}(2\a x)\, dx\nonumber\\
&=\frac{\a^{-z/2}}{2\sqrt{\pi}}\G\left(\frac{z+1}{2}\right)\G\left(1+\frac{z}{2}\right)\sum_{m=1}^{\infty}m^{-z}\int_{0}^{\infty}\sum_{k=1}^{\infty}e^{-2\pi mk\a y}\frac{\sin ((1 + z)\arctan y)\, dy}{(1 + {y^2})^{(1 + z)/2}}\nonumber\\
&=\frac{\a^{z/2}}{2\sqrt{\pi}}\G\left(\frac{z+1}{2}\right)\G\left(1+\frac{z}{2}\right)\sum_{m=1}^{\infty}\int_{0}^{\infty}\frac{\sin\left((1 + z)\arctan\left(\frac{x}{m\a}\right)\right)}{(x^2 + {m^2\a^2})^{(1 + z)/2}}\frac{dx}{e^{2\pi x}-1}\nonumber\\
&=\frac{\a^{z/2}}{2^{z+2}}\G(z+1)\sum_{m=1}^{\infty}\left(\zeta(z+1,m\a)-\frac{(m\a)^{-z}}{2}-\frac{(m\a)^{-z}}{z}\right),
\end{align}
where in the last step, we used \eqref{dup},
and Hermite's formula for the Hurwitz zeta function \cite[p.~609, formula \textbf{25.11.29}]{nist}, namely,
\begin{equation*}
\zeta(w,a)=\frac{1}{2}a^{-w}+\frac{a^{1-w}}{w-1}+2\int_{0}^{\infty}\frac{\sin\left(w\arctan(t/a)\right)}{(a^2+t^2)^{w/2}}\frac{dt}{e^{2\pi t}-1},
\end{equation*}
which is valid for $w\neq 1$ and Re$(a)>0$.

Hence from \eqref{dun1}, \eqref{dun2.5}, and the fact that $\a\b=1$, we deduce that

\begin{align}\label{dun3}
&\frac{2^{z}\pi^{\frac{z-3}{2}}}{\G(z+1)}\int_{0}^{\infty}\Gamma\left(\frac{z-1+it}{4}\right)\Gamma\left(\frac{z-1-it}{4}\right)
\Xi\left(\frac{t+iz}{2}\right)\Xi\left(\frac{t-iz}{2}\right)\frac{\cos\left( \tf{1}{2}t\log\a\right)}{t^2+(z+1)^2}\, dt\nonumber\\
&=\frac{1}{2}\bigg\{\a^{\frac{z+1}{2}}\left(\sum_{n=1}^{\infty}\lambda(n\a, z)-\frac{\zeta(z+1)}{2\alpha^{z+1}}-\frac{\zeta(z)}{\alpha z}\right)+\b^{\frac{z+1}{2}}\left(\sum_{n=1}^{\infty}\lambda(n\b, z)-\frac{\zeta(z+1)}{2\beta^{z+1}}-\frac{\zeta(z)}{\beta z}\right)\bigg\},
\end{align}
where $\lambda(x, z)$ is defined in \eqref{dvarphi}.
To obtain \eqref{hurgeneid}, it only remains to show that 
\begin{align}\label{mainneq22}
\a^{\frac{z+1}{2}}\left(\sum_{n=1}^{\infty}\lambda(n\a, z)-\frac{\zeta(z+1)}{2\alpha^{z+1}}-\frac{\zeta(z)}{\alpha z}\right)=\b^{\frac{z+1}{2}}\left(\sum_{n=1}^{\infty}\lambda(n\b, z)-\frac{\zeta(z+1)}{2\beta^{z+1}}-\frac{\zeta(z)}{\beta z}\right).
\end{align}
The limiting case $z\to 0$ of this identity appears on page $220$ of Ramanujan's Lost Notebook \cite{lnb}, and is discussed in detail in \cite{bcbad}. 

In \cite{dixit} as well as in \cite{transf}, \eqref{mainneq22} was proved as a consequence of Corollary \ref{hurgene} and the fact that $\cos\left(\frac{1}{2}t\log\a\right)=\cos\left(\frac{1}{2}t\log\b\right)$ for $\a\b=1$. Hence to avoid circular reasoning, we must obtain a new proof of it which does not make use of the integral involving the Riemann $\Xi$-function present in this corollary.

The aforementioned limiting case was proved in \cite[Section 4]{bcbad} in the manner sought above using Guinand's generalization of the Poisson summation formula \cite[Theorem 1]{apg1}. This requires use of a result of Ramanujan \cite[Equation (1.4)]{bcbad} that the function $\psi(x+1)-\log x$ is self-reciprocal in the Fourier cosine transform, i.e., 
\begin{equation*}
\int_0^{\i}\left(\psi(1+x)-\log x\right)\cos(2\pi{yx})\, dx
=\df{1}{2}\left(\psi(1+y)-\log y\right).
\end{equation*}
For \eqref{mainneq22}, this method, however, does not look feasible as the one-variable generalization of $\psi(x+1)-\log x$ relevant to the problem, namely $x^{-z}/z-\zeta(z+1,x+1)$, is \emph{not} self-reciprocal in the Fourier cosine transform. Hence we prove \eqref{mainneq22} by first obtaining a new proof of the equivalent modular transformation in the first equality in the following result, also proved in \cite[Theorem 6.3]{dixitmoll} using the integral involving the $\Xi$-function. This new proof is, of course, independent of the use of this integral, and generalizes Koshliakov's proof for the case when $z=0$ \cite[p.~248]{koshli}.

\textit{
Assume $-1 <$ \textup{Re} $z < 1$. Let $\Omega(x, z)$ be defined by
\begin{equation*}
\Omega(x,z) := 2 \sum_{n=1}^{\infty} \sigma_{-z}(n) n^{z/2} 
\left( e^{\pi i z/4} K_{z}( 4 \pi e^{\pi i/4} \sqrt{nx} ) +
 e^{-\pi i z/4} K_{z}( 4 \pi e^{-\pi i/4} \sqrt{nx} ) \right),
\end{equation*}
where $\sigma_{-z}(n)=\sum_{d|n}d^{-z}$ and $K_\nu(z)$ denotes the modified Bessel function of order $\nu$. Then for $\alpha, \beta>0, \alpha\beta=1$, 
\begin{align}\label{genelkosh}
&\alpha^{(z+1)/2} 
\int_{0}^{\infty} e^{-2 \pi \alpha x} x^{z/2} 
\left( \Omega(x,z) - \frac{1}{2 \pi} \zeta(z) x^{z/2-1} \right)\, dx\nonumber \\
&=\beta^{(z+1)/2} 
\int_{0}^{\infty} e^{-2 \pi \beta x} x^{z/2} 
\left( \Omega(x,z) - \frac{1}{2 \pi} \zeta(z) x^{z/2-1} \right)\, dx.
\end{align}}

Upon proving \eqref{genelkosh}, we show that
\begin{align}\label{equi}
&\int_{0}^{\infty} e^{-2 \pi \alpha x} x^{z/2} \left( \Omega(x,z) - \frac{1}{2 \pi} \zeta(z) x^{z/2-1} \right) dx\nonumber\\
&=\frac{\G(z+1)}{(2\pi)^{z+1}}\sum_{n=1}^{\infty}\left(\lambda(n\a, z)-\frac{\zeta(z+1)}{2\alpha^{z+1}}-\frac{\zeta(z)}{\alpha z}\right),
\end{align}
thereby proving \eqref{mainneq22}. The special case of \eqref{equi} was obtained in \cite{ingenious}.

The function $\Omega(x,z)$ has many nice properties. For example, it has a very useful inverse Mellin transform representation \cite[Equation (6.6)]{dixitmoll}, valid for $c=$ Re $s>1\pm$ Re $\frac{z}{2}$:
\begin{equation}
\label{Mellin-omega}
\Omega(x,z) = \frac{1}{2 \pi i} 
\int_{c - i \infty} ^{c + i \infty} 
\frac{\zeta(1 - s + \tfrac{z}{2}) \zeta(1 -s  - \tfrac{z}{2} ) }
{2 \cos\left( \tfrac{1}{2} \pi \left( s + \tfrac{z}{2} \right)\right)} x^{-s} ds.
\end{equation}
It also satisfies the following identity \cite[Proposition 6.1]{dixitmoll} and Re $x>0$:
\begin{equation}\label{omegarep}
\Omega(x,z) = - \frac{\Gamma(z) \zeta(z)}{(2 \pi \sqrt{x})^{z}}  +
\frac{x^{z/2-1}}{2 \pi} \zeta(z) - 
\frac{x^{z/2}}{2} \zeta(z+1)  +
\frac{x^{z/2+1}}{\pi} \sum_{n=1}^{\infty} \frac{\sigma_{-z}(n)}{n^{2}+x^{2}}.
\end{equation}
Lastly, we mention that it plays a vital role in obtaining a very short proof of the extended version of the Vorono\"{\dotlessi} summation formula \cite[Theorem 6.1]{bdrz} in the case when the associated function is analytic in a specific region.

We begin with the following lemma, which is interesting in its own right, and shows that the function $\Omega(x,z)-\frac{1}{2\pi}\zeta(z)x^{\frac{z}{2}-1}$ is self-reciprocal when integrated against the Bessel function of the first kind of order $z$. It is a one variable generalization of a result of Koshliakov \cite[Equation (11)]{koshlisum}.
\begin{lemma}\label{selfomega}
Let $J_{\nu}(w)$ denote the Bessel function of the first kind of order $\nu$.  Let $-1 <$ \textup{Re} $z < 1$. For $\textup{Re}(x)>0$, we have
\begin{equation*}
\int_{0}^{\infty}J_{z}(4\pi\sqrt{xy})\left(\Omega(y,z)-\frac{1}{2\pi}\zeta(z)y^{\frac{z}{2}-1}\right)\, dy=\frac{1}{2\pi}\left(\Omega(x,z)-\frac{1}{2\pi}\zeta(z)x^{\frac{z}{2}-1}\right).
\end{equation*}
\end{lemma}
\begin{proof}
For $-$ Re $\left(\frac{z}{2}\right)<$ \textup{Re} $s<\frac{3}{4}$, we have \cite[p.~225]{dixitmoll}
\begin{equation*}
\int_{0}^{\infty} x^{s-1} J_{z}( 4 \pi \sqrt{xy} ) dx = 
2^{-2s} \pi^{-2s} y^{-s} 
\frac{\Gamma \left( s + \tfrac{z}{2} \right)}
{\Gamma \left( 1 - s + \tfrac{z}{2} \right)},
\end{equation*}
and from \eqref{Mellin-omega} and an application of the residue theorem, we find that for Re $s<1\pm$ Re $\left(\frac{z}{2}\right)$,
\begin{equation}\label{Mellin-omega1}
\int_{0}^{\infty}y^{s-1}\left(\Omega(y,z)-\frac{1}{2\pi}\zeta(z)y^{\frac{z}{2}-1}\right)=\frac{\zeta(1 - s + \tfrac{z}{2}) \zeta(1 -s  - \tfrac{z}{2} ) }
{2 \cos\left( \tfrac{1}{2} \pi \left( s + \tfrac{z}{2} \right)\right)}.
\end{equation}
Hence using Parseval's formula \cite[p.~83, Equation (3.1.11)]{kp}, we see that for $\pm$ \textup{Re} $\left(\frac{z}{2}\right)<c=$ \textup{Re} $s<\min\left(\frac{3}{4}, 1\pm\text{Re}\left(\frac{z}{2}\right)\right)$,
\begin{align*}
&\int_{0}^{\infty}J_{z}(4\pi\sqrt{xy})\left(\Omega(y,z)-\frac{1}{2\pi}\zeta(z)y^{\frac{z}{2}-1}\right)\, dy\nonumber\\
&=\frac{1}{2\pi i}\int_{(c)}\frac{(2\pi)^{-2s}x^{-s}\G\left(s+\frac{z}{2}\right)}{\G\left(1-s+\frac{z}{2}\right)}\frac{\zeta\left(s-\frac{z}{2}\right)\zeta\left(s+\frac{z}{2}\right)}{2\sin\left(\frac{\pi}{2}\left(s-\frac{z}{2}\right)\right)}\, ds\nonumber\\
&=\frac{1}{2\pi i}\int_{(c)}\frac{\zeta(1 - s + \tfrac{z}{2}) \zeta(1 -s  - \tfrac{z}{2} ) }
{4\pi \cos\left( \tfrac{1}{2} \pi \left( s + \tfrac{z}{2} \right)\right)} x^{-s} ds,
\end{align*}
where in the last step we used the functional equation for $\zeta(s)$ twice. The result now follows from \eqref{Mellin-omega1}.
\end{proof}
We now prove the modular transformation \ref{genelkosh}. Using Lemma \ref{selfomega}, we see that
\begin{align}\label{inttra}
&\alpha^{(z+1)/2} 
\int_{0}^{\infty} e^{-2 \pi \alpha x} x^{z/2} 
\left( \Omega(x,z) - \frac{1}{2 \pi} \zeta(z) x^{z/2-1} \right) dx\nonumber \\
&=2\pi\alpha^{(z+1)/2} \int_{0}^{\infty} e^{-2 \pi \alpha x} x^{z/2} \int_{0}^{\infty}J_{z}(4\pi\sqrt{xy})\left(\Omega(y,z)-\frac{1}{2\pi}\zeta(z)y^{\frac{z}{2}-1}\right)\, dy\, dx\nonumber\\
&=2\pi\alpha^{(z+1)/2}\int_{0}^{\infty}\left(\Omega(y,z)-\frac{1}{2\pi}\zeta(z)y^{\frac{z}{2}-1}\right)\int_{0}^{\infty} e^{-2 \pi \alpha x} x^{z/2}J_{z}(4\pi\sqrt{xy})\, dx\, dy,
\end{align}
where the interchange of the order of integration follows from Fubini's theorem.

Now use the formula \cite[p.~709, formula \textbf{6.643.1}]{gr}
\begin{equation*}
\int_{0}^{\infty}x^{\mu-\frac{1}{2}}e^{-ax}J_{2\nu}(2b\sqrt{x})\, dx=\frac{\G\left(\mu+\nu+\frac{1}{2}\right)}{b\G(2\nu+1)}e^{-\frac{b^2}{2a}}a^{-\mu}M_{\mu,\nu}\left(\frac{b^2}{a}\right),
\end{equation*}
which is valid for Re $\left(\mu+\nu+\frac{1}{2}\right)>0$, with $b=2\pi\sqrt{y}$, $\nu=z/2, \mu=(z+1)/2, a=2\pi\a$, and then the definition \cite[p.~1024]{gr} of the Whittaker function
\begin{equation}
M_{k,\mu}(z)=z^{\mu+\frac{1}{2}}e^{-z/2}{}_1F_{1}\left(\mu-k+\frac{1}{2};2\mu+1;z\right),
\end{equation}
to deduce that
\begin{equation}\label{intm}
\int_{0}^{\infty} e^{-2 \pi \alpha x} x^{z/2}J_{z}(4\pi\sqrt{xy})\, dx=\frac{e^{-2\pi y/\a}y^{z/2}}{2\pi\a^{z+1}}.
\end{equation}
Finally we obtain \eqref{genelkosh} from \eqref{inttra}, \eqref{intm} and the fact that $\a\b=1$.

It only remains to now prove \eqref{equi}. We first prove it for $0<$ Re $z<1$ and then extend it by analytic continuation. To that end, we use \eqref{omegarep} to evaluate the integral on the left-hand side of \eqref{equi}. Note that
\begin{equation}\label{genf-1}
\int_{0}^{\infty} e^{-2 \pi \alpha x} x^{z/2}\left(- \frac{\Gamma(z) \zeta(z)x^{-z/2}}{(2 \pi )^{z}}- \frac{x^{z/2}}{2} \zeta(z+1)\right)\, dx=-\frac{\G(z)\zeta(z)}{\a(2\pi)^{z+1}}-\frac{\G(z+1)\zeta(z+1)}{2(2\pi\a)^{z+1}}.
\end{equation}
Also,
\begin{align*}
\int_{0}^{\infty} e^{-2 \pi \alpha x} \frac{x^{z+1}}{\pi}\sum_{n=1}^{\infty}\frac{\sigma_{-z}(n)}{n^2+x^2}\, dx=\sum_{m=1}^{\infty}m^{-z}\int_{0}^{\infty}\frac{x^{z+1}}{\pi}e^{-2\pi\a x}\sum_{k=1}^{\infty}\frac{1}{x^2+m^2k^2}\, dx,
\end{align*}
where the interchange of the order of summation and integration is justified by absolute convergence. It is well known \cite[p.~191]{con} that for $t\neq 0$,
\begin{equation*}
2t\sum_{k=1}^{\infty}\frac{1}{t^2+4k^2\pi^2}=\frac{1}{e^{t}-1}-\frac{1}{t}+\frac{1}{2}.
\end{equation*}
So
\begin{align}\label{ds1}
\int_{0}^{\infty}e^{-2\pi\a x}\frac{x^{z+1}}{\pi}\sum_{n=1}^{\infty}\frac{\sigma_{-z}(n)}{x^2+n^2}\, dx&=\sum_{m=1}^{\infty}m^{-z}\int_{0}^{\infty}e^{-2\pi\a x}x^{z}\left(\frac{1}{m(e^{2\pi x/m}-1)}-\frac{1}{2\pi x}+\frac{1}{2m}\right)\, dx\nonumber\\
&=\sum_{m=1}^{\infty}\int_{0}^{\infty}e^{-2\pi\a mt}t^{z}\left(\frac{1}{e^{2\pi t}-1}-\frac{1}{2\pi t}+\frac{1}{2}\right)\, dt\nonumber\\
&=\frac{\G(z+1)}{(2\pi)^{z+1}}\sum_{m=1}^{\infty}\left\{\zeta(z+1,m\a)-\frac{(m\a)^{-z}}{z}-\frac{(\a m)^{-z-1}}{2}\right\}\nonumber\\
&=\frac{\G(z+1)}{(2\pi)^{z+1}}\sum_{m=1}^{\infty}\lambda(m\a, z),
\end{align}
where in the penultimate step, we used the following result \cite[p.~609, formula (25.11.27)]{nist}, valid for Re$(w)>-1, w\neq 1$, Re$(a)>0$:
\begin{equation*}
\zeta(w,a)=\frac{1}{2}a^{-w}+\frac{a^{1-w}}{w-1}+\frac{1}{\G(w)}\int_{0}^{\infty}x^{w-1}e^{-ax}\left(\frac{1}{e^x-1}-\frac{1}{x}+\frac{1}{2}\right)\, dx
\end{equation*}
with $w=z+1, a=\a m$ amd $x=2\pi t$.

Finally \eqref{genf-1} and \eqref{ds1} give \eqref{equi}. This completes the proof of the modular transformation in \eqref{mainneq22} for $0<$ Re $z<1$. As explained in \cite[p.~1162]{dixit}, the result follows for $-1<$ Re $z<1$ by analytic continuation.

This proves \eqref{mainneq22}, and hence along with \eqref{dun3}, completes the proof of Corollary \ref{hurr}.

\begin{center}
\textbf{Acknowledgements}
\end{center}
The first author is funded in part by the grant NSF-DMS 1112656 of Professor Victor H. Moll of Tulane University and sincerely thanks him for the support. The second author wishes to acknowledge partial support of SNF grant 200020\_149150$\backslash$1.

\end{document}